\documentclass[final,leqno,onefignum,onetabnum]{siamltex1213}
\usepackage{graphicx}
\usepackage{multirow}
\usepackage{amssymb}

\title{Numerical evaluation of two and three parameter Mittag--Leffler functions\thanks{This work has been supported under the GNCS - INdAM Project 2014.}} 

\author{
Roberto Garrappa\thanks{Universit\`a degli Studi di Bari ``Aldo Moro'', Dipartimento di Matematica  Via E. Orabona n.4, 70125 Bari, Italy (\email{roberto.garrappa@uniba.it})}}

\begin{document}
\maketitle
\slugger{sinum}{xxxx}{xx}{x}{x--x}

\begin{abstract}
The Mittag-Leffler (ML) function plays a fundamental role in fractional calculus but very few methods are available for its numerical evaluation. In this work we present a method for the efficient computation of the ML function based on the numerical inversion of its Laplace transform (LT): an optimal parabolic contour is selected on the basis of the distance and the strength of the singularities of the LT, with the aim of minimizing the computational effort and reduce the propagation of errors. Numerical experiments are presented to show accuracy and efficiency of the proposed approach. The application to the three parameter ML (also known as Prabhakar) function is also presented.
\end{abstract}



\pagestyle{myheadings}
\thispagestyle{plain}
\markboth{R. GARRAPPA}{NUMERICAL EVALUATION OF ML FUNCTIONS}

\newtheorem{remark}[theorem]{Remark}
\def\Rset{{\mathbb{R}}}
\def\Cset{{\mathbb{C}}}
\def\Zset{{\mathbb{Z}}}
\def\Nset{{\mathbb{N}}}
\def\Arg{\mathop{\rm Arg}}

\section{Introduction}

The Mittag--Leffler (ML) function was introduced, at the beginning of the twentieth century, by the Swedish mathematician Magnus Gustaf Mittag--Leffler \cite{Mittag-Leffler1902,MittagLeffler1904} while studying summation of divergent series. Extensions to two \cite{Wiman1905} and three \cite{Prabhakar1971} parameters of the original single parameter function were successively considered; all these functions can be regarded as special instances of the generalized hypergeometric function investigated by Fox \cite{Fox1928} and Wright \cite{Wright1935}. 

Until the 1960s, few authors (e.g., \cite{HilleTamarkin1930}) recognized the importance of the ML function in fractional calculus and, in particular, in describing anomalous processes with hereditary effects \cite{AlBassam1965,CaputoMainardi1971,CaputoMainardi2007,DzherbashyanNersesyan1968}. For an historical outline and a review of the main properties of the ML function we refer to \cite{HauboldMathaiSaxena2011,MainardiGorenflo2000} and to the recent monograph \cite{GorenfloKilbasMainardiRogosin2014}.


For any argument $z \in \Cset$, the ML function with two parameters $\alpha,\beta \in \Cset$, with $\Re(\alpha)>0$, is defined by means of the series expansion
\begin{equation}\label{eq:ClassicalML}
	E_{\alpha,\beta}(z) = \sum_{k=0}^{\infty} \frac{z^{k} }{\Gamma(\alpha k + \beta)} ,
\end{equation}
where $\Gamma(z) = \int_{0}^{\infty} t^{z-1} e^{-t} dt$ is the Euler's gamma function; since the integral representation of $\Gamma(z)$ holds only for $\Re(z)>0$, the extension to the half--plane $\Re(z) \le 0$, with $z \not\in \{0,-1,-2,\dots\}$, is accomplished by means of the relationship $\Gamma(z+n)=z(z+1)\cdots(z+n-1)\Gamma(z)$, $n \in \Nset$, \cite{KilbasSrivastavaTrujillo2006,Mainardi2010}. 

As a special case, the ML function with one parameter is obtained for $\beta=1$, i.e. $E_{\alpha}(z)=E_{\alpha,1}(z)$, whilst the generalization to a third parameter $\gamma$ 
\begin{equation}\label{eq:PrabhakarFunction}
	E_{\alpha,\beta}^{\gamma} (z) = \frac{1}{\Gamma(\gamma)}\sum_{k=0}^{\infty} \frac{\Gamma(\gamma+k) z^{k}}{k!\Gamma(\alpha k + \beta)}
\end{equation}
is recently receiving an increasing attention due to the applications in modeling polarization processes in anomalous or inhomogeneous materials \cite{OliveiraMainardiVaz2011}. 

In this work we restrict our attention to real parameters $\alpha$, $\beta$ and $\gamma$, with $\alpha>0$ and $\gamma>0$, since they are of more practical interest. 


With the exception of a few special cases in which the ML function can be represented in terms of other elementary and special functions, as for instance $E_{1,1}(z) = e^{z}$, $E_{2,1}(z^{2}) = \cosh(z)$, $E_{2,1}(-z^{2}) = \cos(z)$ and $E_{\frac{1}{2},1}(\pm z^{1/2}) = e^{z} \textrm{erfc}(\mp z^{1/2})$, most of the programming languages do not provide built-in functions for the ML function. 

Although it is theoretically possible to evaluate $E_{\alpha,\beta}(z)$ and $E_{\alpha,\beta}^{\gamma}(z)$ by truncating the series in (\ref{eq:ClassicalML}) and (\ref{eq:PrabhakarFunction}), in the majority of the applications this is not advisable: with arguments having moderate or large modulus $|z|$ the convergence of the series is very slow and involves an exceedingly large amount of computation; moreover, a large number of terms in the series can significantly grow before decreasing, thus generating overflow or numerical cancellation unless variable precision arithmetic is used. In finite precision arithmetic (which is the natural environment for scientific computing) the use of (\ref{eq:ClassicalML}) and (\ref{eq:PrabhakarFunction}) is, therefore, confined just to small arguments.

Very few methods have been presented so far in the literature. The sophisticated algorithm described in \cite{GorenfloLoutchkoLuchko2002} uses different techniques to evaluate the ML function and its derivative in different parts of the complex plane. Other approaches based on mixed techniques (Taylor series, asymptotic series, and integral representations) were discussed in \cite{HilferSeybold2006,HilferSeybold2008}. The only existing Matlab code \cite{PodlubnyKacenak2012} (which implements some of the ideas introduced in \cite{GorenfloLoutchkoLuchko2002}) shows a great variability in the amount of computation required to achieve a prescribed accuracy and in some regions of the complex plane turns out to be poorly accurate. 

The recent introduction \cite{GarrappaMoretPopolizio2014_JCP,GarrappaPopolizio2011_CAMWA,Moret2013} of new methods for fractional differential equations involving a large number of evaluations of the ML function, also with matrix arguments \cite{MoretNovati2011}, motivates the investigation of different techniques to perform the computation, over the whole complex plane, in an accurate and fast way.

In this paper we consider an approach based on the inversion of the Laplace transform (LT) in which a quadrature rule is applied on a suitable complex contour, namely a parabola. Methods of this kind have been successfully applied \cite{GarrappaPopolizio2013,WeidemanTrefethen2007} to the ML function restricted to some very special cases ($0<\alpha<1$, $\beta=1$ or real $z$). 

The extension to the more general case is however not trivial and demands not only a different and more in-depth theoretical analysis but also a thorough different strategy. Since the possible presence of a large number of singularities of the LT, it can indeed be impossible to find a contour encompassing all the singularities and behaving in a satisfactory way for computational purposes. Our approach is therefore to consider separate regions in which the LT is analytic and look, in each region, for the contour and discretization parameters allowing to achieve a given accuracy. The \emph{optimal parabolic contour} (OPC) algorithm hence selects the region in which the numerical inversion of the LT is actually performed by choosing the one in which both the computational effort and the errors are minimized.

The paper is organized as follows. Section \ref{S:MLLaplace} introduces the LT of the ML function, describes its analyticity properties and discusses the numerical inversion. In Section \ref{S:Parabolic} the OPC algorithm is presented and a detailed error analysis is derived in order to provide information for the selection of the optimal contour and of the suitable quadrature parameters. Section \ref{S:Numerical} is hence devoted to illustrate numerical experiments and some concluding remarks are discussed in Section \ref{S:Conclusion}.

\section{Evaluation of the ML function by LT inversion}\label{S:MLLaplace}

During the last decades, an increasing amount of attention has been devoted to methods for computing special functions by inverting the LT; there are some key factors accounting for this interest:
\begin{enumerate}
	\item for several functions (including the ML) the LT has an analytical formulation which is much more simple than the function itself;
	\item algorithms for the numerical inversion of the LT are usually quite simple to implement and run in a fast way;
	\item it is possible to derive accurate error estimations and perform the computation virtually within any prescribed accuracy.
\end{enumerate}

Although from a theoretical point of view the inversion of the LT is an ill--posed problem, satisfactory numerical results are expected for the ML function since it is possible to evaluate its LT in the whole complex plane and with high accuracy.


An explicit representation of the LT of (\ref{eq:ClassicalML}) and (\ref{eq:PrabhakarFunction}) is, however, not available. We must therefore introduce the following generalization of the ML function (\ref{eq:ClassicalML}) 
\begin{equation}\label{eq:GeneralizedML}
	e_{\alpha,\beta}(t;\lambda) = t^{\beta-1} E_{\alpha,\beta}(t^{\alpha} \lambda) , \quad
	t \in \Rset_{+}, \ \lambda \in \Cset ,
\end{equation}
in order to express the corresponding LT as \cite{KilbasSrivastavaTrujillo2006,Mainardi2010,Podlubny1999} 
\[
	{\cal E}_{\alpha,\beta}(s;\lambda) = \frac{s^{\alpha-\beta}}{s^{\alpha}-\lambda}
	, \quad \Re(s) > 0 \textrm{ and } |\lambda s^{-\alpha}| < 1 ,
\]
(for easy of presentation we just focus on the two parameter function (\ref{eq:ClassicalML}); the extension to three parameter case (\ref{eq:PrabhakarFunction}) will be discussed in the Subsection \ref{SS:Prabhakar}).


By means of the formula for the inversion of the LT it is possible to formulate the following integral representation of  $e_{\alpha,\beta}(t;\lambda)$ 
\begin{equation} \label{eq:MLInverseLaplaceBromwich}
	e_{\alpha,\beta}(t;\lambda) = 
	\frac{1}{2\pi i} \int_{\sigma-i\infty}^{\sigma+i\infty} e^{st} {\cal E}_{\alpha,\beta}(s;\lambda) ds,
\end{equation}
where $(\sigma-i\infty , \sigma+i\infty$) is the Bromwich line, with $\sigma\in\Rset$ chosen to ensure that all the singularities of ${\cal E}_{\alpha,\beta}(s;\lambda)$ lie to the left of the line $\Re(s) = \sigma$. Since the presence of non integer powers, ${\cal E}_{\alpha,\beta}(s;\lambda)$ is a multi-valued function and a branch-cut extending from $0$ to $-\infty$ along the real axis is introduced to make the integrand single-valued. 

\begin{remark}\label{rmk:lambdanot0}
For convenience we assume $\lambda\not=0$; it is readily verified that $e_{\alpha,\beta}(t;0) = t^{\beta-1}/\Gamma(\beta)$. Moreover, also $t=0$ is of no interest since $e_{\alpha,\beta}(0;\lambda)=0$ for $\beta>1$, $e_{\alpha,1}(0;\lambda)=1/\Gamma(\beta)$ and $e_{\alpha,\beta}(t;\lambda)\to+\infty$ as $t\to0_{+}$ for $\beta < 1$. 
\end{remark}


As first suggested by Talbot \cite{Talbot1979}, to exploit (\ref{eq:MLInverseLaplaceBromwich}) for numerical computation it is necessary to deform the Bromwich line into an equivalent contour ${\cal C}$ that begins and ends in the left half of the complex plane in order to rapidly dampen the exponential factor $e^{st}$ and avoid high oscillations which are source of numerical instability (for the equivalence of the contours it is meant that they encompass all the singularities of ${\cal E}_{\alpha,\beta}(s;\lambda)$ to the left). Once a suitable contour is selected, a quadrature rule can be applied.

The above two steps are intimately related. As deeply studied in \cite{TrefethenWeideman2014,WeidemanTrefethen2007}, the choice of the contour affects in a significant way the convergence properties of the quadrature rule which depend on the analyticity of the integrand in a region surrounding the path of integration. A satisfactory selection of the deformed contour is therefore not possible without a subtle analysis of the regions in which ${\cal E}_{\alpha,\beta}(s;\lambda)$ is analytic.

After denoting $\theta = \Arg(\lambda)$, $-\pi < \theta \le \pi$, the poles of ${\cal E}_{\alpha,\beta}(s;\lambda)$, i.e. the solutions of the equation $s^{\alpha} - \lambda = 0$, are 
\begin{equation}\label{eq:EquationSingularity}
	\bar{s}^{\star}_j = \lambda^{1/\alpha} = | \lambda |^{1/\alpha} e^{i\frac{\theta + 2j\pi}{\alpha}} 
	, \quad j \in \Zset .
\end{equation}

The relevant poles are those in the main Riemann sheet, for which it is $- \pi < (\theta + 2j\pi)/\alpha \le \pi$ or, equivalently, such that $j$ belongs to
\begin{equation}\label{eq:BoundsBrunchCut}
	\bar{J}(\alpha,\theta) = 
	\left\{j\in \Zset \, \Bigr| \, -\frac{\alpha}{2} - \frac{\theta}{2\pi} < j \le \frac{\alpha}{2} - \frac{\theta}{2\pi} \right\} ;
\end{equation}
their number depends on $\alpha$ and $\theta$, ranging from zero when $0<\alpha<1$ and $|\theta|>\alpha\pi$ to a possible very large number otherwise. 

The origin is a pole only when $\beta > \alpha$; however, it must be always included among the singularities because a branch--point singularity occurs at the origin. 

From a formal point of view we denote with $S^{\star}=\bigl\{s_{0}^{\star},s_{1}^{\star},\dots,s_{J}^{\star}\bigr\}$ the set of all singularities (the poles and the branch--point) of ${\cal E}_{\alpha,\beta}(s;\lambda)$, where $s_{0}^{\star}=0$ and $s_{j}^{\star}=\bar{s}_{j-1-\min \bar{J}(\alpha,\theta)}^{\star}$, $j=1,\dots,J$ with $J=|\bar{J}(\alpha,\theta)|$.

In the presence of a large number of singularities, or when some of them have large imaginary part, it can result nearly impossible to find suitable contours allowing a fast decay of the exponential factor and, at the same time, encompassing all the singularities. For this reason it can be useful, thanks to the Cauchy's residue theorem, to remove some of the poles by residue subtraction
\begin{equation} \label{eq:GeneralizedMLInverseLaplace}
	e_{\alpha,\beta}(t;\lambda) = 
		\sum_{s^{\star} \in S^{\star}_{\cal C}} \mathop{Res} \bigl( e^{st} {\cal E}_{\alpha,\beta}(s;\lambda) , s^{\star} \bigr) +
		\frac{1}{2\pi i} \int_{{\cal C}} e^{st} {\cal E}_{\alpha,\beta}(s;\lambda) ds,
\end{equation}
where $S^{\star}_{\cal C} \subseteq S^{\star}$ is the set of all singularities of ${\cal E}_{\alpha,\beta}$ laying on the rightmost part of the complex plane delimited by ${\cal C}$ and $\mathop{Res} \bigl( f , s^{\star} \bigr)$ denotes the residue of the function $f$ at $s^{\star}$ (observe that, due to the selected branch--cut, ${\cal C}$ cannot traverse the negative real semi axis and must encompass at least $s^{\star}_0=0$ to its left).

It is a favorable achievement that the residues in (\ref{eq:GeneralizedMLInverseLaplace}) can be explicitly represented in terms of elementary functions as
\[
	\mathop{Res} \bigl( e^{st} {\cal E}_{\alpha,\beta}(s;\lambda) , s^{\star} \bigr)
	= \frac{1}{\alpha} \bigl(s^{\star} \bigr)^{1-\beta} e^{s^{\star} t} .
\]

Assumed the contour ${\cal C}$ be represented by means of a complex--valued function $z(u)$, $-\infty < u < \infty$, 
then (\ref{eq:GeneralizedMLInverseLaplace}) can be rewritten as
\begin{equation}\label{eq:GeneralizedMLInverseLaplace2}
	e_{\alpha,\beta}(t;\lambda) = 
		 \frac{1}{\alpha} \sum_{s^{\star} \in S^{\star}_{\cal C}} \bigl(s^{\star} \bigr)^{1-\beta} e^{s^{\star} t} +
		\frac{1}{2\pi i} \int_{-\infty}^{\infty} g(u) du,
\end{equation}
where
\begin{equation}\label{eq:IntegrandFunction}
	g(u) = e^{z(u)t} {\cal E}_{\alpha,\beta}(z(u);\lambda) z'(u)
	= \frac{e^{z(u)t}  (z(u))^{\alpha-\beta}  z'(u)}{(z(u))^{\alpha} - \lambda } .
\end{equation}

Numerical quadratures for integrals on unbounded intervals 
\[
	I = \int_{-\infty}^{\infty} g(u) du
\]
are presented in several papers and reference books (e.g., see \cite{DavisRabinowitz1984}). An extensive analysis of the trapezoidal rule has been recently provided in the remarkable paper by Trefethen and Weideman \cite{TrefethenWeideman2014} which not only focuses on the fast convergence of the trapezoidal rule but also discusses its main practical applications. Despite its simplicity, the trapezoidal rule appears indeed as a powerful tool to perform fast and highly accurate integration in a variety of applications.

On a given equispaced grid $kh$, $k\in\Zset$, with step-size $h>0$, the infinite and finite trapezoidal approximations of $I$ are
\[
	I_{h} = h  \sum_{k=-\infty}^{\infty} g(kh) , \quad	
	I_{h,N} = h  \sum_{k=-N}^{N} g(kh)
\]
and the corresponding error $I - I_{h,N}$ results from the sum of the discretization error $DE = \bigl|I-I_{h}\bigr|$ and the truncation error $TE = \bigl|I_{h}-I_{h,N}\bigr|$.

Under the assumption that $g(u)$ decays rapidly as $u \to \pm \infty$, an estimation of $TE$ is given by the last term retained in the summation, i.e. $TE = {\cal O} \bigl(\left|g(hN)\right|\bigr)$, $N\to \infty$.

As discussed in \cite{TrefethenWeideman2014,WeidemanTrefethen2007}, the estimation of $DE$ is performed on the basis of the analyticity properties of $g(u)$. For reasons which will be clear later, we need to slightly modify the statement of the original result, with no substantial changes in the proof which remains the same as outlined in \cite{TrefethenWeideman2014}.

\begin{theorem}\label{thm:ErrorIntegrationContour}
Let $g(w)$ be analytic in a strip $-d^{\star}<\Im(w)<c^{\star}$, for some $c^{\star}>0$, $d^{\star}>0$, with $g(w)\to0$ uniformly as $|w|\to\infty$ in that strip. For any $0<c<c^{\star}$ and $0<d<d^{\star}$ it is
\[
	DE = \bigl| I - I_{h} \bigr| \le DE_{+}(c) + DE_{-}(d) ,
\]
where 
\[
	DE_{+}(c) = \frac{M_{+}(c)}{e^{2\pi c/h} - 1}
	, \quad
	DE_{-}(d) = \frac{M_{-}(d)}{e^{2\pi d/h} - 1} 
\]
and
\[
 M_{+}(c) = \max_{0\le r \le c} \int_{-\infty}^{\infty} \bigl| g(u+ir) \bigr| du  
 , \quad
 M_{-}(d) = \max_{0\le r \le d} \int_{-\infty}^{\infty} \bigl| g(u-ir) \bigr| du .
\]
\end{theorem}

In most cases (for instance with the exponential function \cite{WeidemanTrefethen2007}),  the contribution of $M_{+}(c)$ and $M_{-}(d)$ is negligible and the estimations $DE_{+}(c) \approx e^{-2\pi c/h}$ and $DE_{+}(d) \approx e^{-2\pi d/h}$ are sufficiently accurate for a satisfactory error analysis. When applied to the ML function it is possible, depending on the parameters $\alpha$ and $\beta$, that $M_{+}(c)\to+\infty$ when $c \to c^{\star}$ and $M_{-}(d)\to+\infty$ when $d \to d^{\star}$. The consequence of unbounded limits for $M_{+}(c)$ and $M_{-}(d)$ is that their contribution can be nonnegligible. This is especially true within very narrow strips of analyticity (as when there are several singularities), for which $c$ or $d$ are necessarily close to their upper bounds $c^{\star}$ and $d^{\star}$.

Providing a reliable estimation for $M_{+}(c)$ and $M_{-}(d)$ (and for the rate by which they tend to $+\infty$) and including them in the error analysis is therefore of utmost importance in order to select optimal parameters and fulfill an assigned accuracy.



\section{Parabolic contours and the OPC algorithm}\label{S:Parabolic}

Removing some of the singularities by the residue subtraction in (\ref{eq:GeneralizedMLInverseLaplace}) offers a considerable freedom in the choice of the integration path. The task of selecting a suitable contour in a specific region of the complex plane is greatly simplified by first fixing the geometric shape and hence adopting a parametrized description of the contour with very few (usually just one) parameters; the problem is thus reduced to the evaluation of the optimal parameters.


Several families of contours have been proposed so far. After the original work of Talbot \cite{Talbot1979} on contours of cotangent shape (see also  \cite{DingfelderWeideman2014,MurliRizzardi1990,Weideman2006}), a special attention has been paid to parabolic \cite{Butcher1957,GavrilyukMakarov2005,Weideman2010,WeidemanTrefethen2007} and hyperbolic contours \cite{LopezPalenciaSchadle2006,GavrilyukMakarov2005,SheenSloanThomee2003,WeidemanTrefethen2007}. 

The convergence rates of the $N$--points trapezoidal rule on cotangent, hyperbolic and parabolic contours have been studied in \cite{TrefethenWeidemanSchmelzer2006}; the respective rates of ${\cal O}\bigl(3.89^{-N}\bigr)$, ${\cal O}\bigl(3.20^{-N}\bigr)$ and ${\cal O}\bigl(2.85^{-N}\bigr)$ indicate a fast convergence with all these contours.  


Although the convergence with cotangent and hyperbolic contours is slightly faster, the simpler representation of parabolic contours makes them much more easy to handle; therefore, parabolas appear to be preferable especially when the presence of a certain number of singularities demands the fulfillment of tightened constraints. 

As in \cite{TrefethenWeideman2014,WeidemanTrefethen2007}, for a real parameter $\mu>0$ we consider the parabolic contour
\begin{equation}\label{eq:ParabolicContour}
        {\cal C} \, : \, z(u) = \mu \bigl( i u + 1 \bigr)^2
        , \quad
        - \infty < u < \infty .
\end{equation}

To select the singularities that must be removed in (\ref{eq:GeneralizedMLInverseLaplace}), we partition the complex plane in neighboring regions having the singularities of ${\cal E}_{\alpha,\beta}$ on their respective boundaries; in each region the parabolic contour and the discretization parameters are determined, according to a suitably modified version of the procedure described in \cite{GarrappaPopolizio2013,WeidemanTrefethen2007}, with the aim of fulfilling a prescribed accuracy $\varepsilon>0$. Among the possible contours (one for each region), the OPC algorithm makes an optimal selection with respect to the computational effort: the region and the contour involving the smaller number $N$ of quadrature nodes is selected. Nevertheless, some issues related to reduce the propagation of round-off errors are also addressed.

As already observed in \cite{GarrappaPopolizio2013}, the computation necessary to select the contour is much less than the computation involved by the inversion of the LT. Thus the overall process of establishing the  contour in an optimal way adds only a negligible amount of computation, with the obvious advantage of performing the actual, and more expensive, inversion with the smallest possible number of floating point operations. 

The starting step in the OPC algorithm is to sort the singularities of ${\cal E}_{\alpha,\beta}$ in order to identify a sequence of regions delimited by two consecutive singularities. To this purpose we introduce the function $\varphi : \Cset \to \Rset_{+}$ defined according to
\[
	\varphi(s) = \frac{\Re (s) + |s|}{2} .
\]

The function $\varphi$ allows to split the complex plane in regions bounded by parabolas of type (\ref{eq:ParabolicContour}) as stated in the following Proposition.

\begin{proposition}\label{prop:PositionPointParabola}
Let $z(u) = \mu (iu+1)^2$, with $\mu > 0$. A point $s\in \Cset$ lies on the parabola $z(u)$ whenever $\varphi(s) = \mu$. Moreover, whenever $\varphi(s) < \mu$ the point $s$ lies at the left of the parabola $z(u)$ and whenever $\varphi(s) > \mu$ the point $s$ lies at the right of the parabola $z(u)$.
\end{proposition}

\begin{proof}
After expanding $z(u) = \mu (1-u^2) + 2 i u \mu$, it is immediate to see that a point $s \in \Cset$ lies on the parabola described by $z(u)$ whenever
\[
	\left\{\begin{array}{l}
		\mu \left( 1 - u^2 \right) = \Re (s) \\
		2 \mu u = \Im (s) \\
	\end{array}; \right.
\]
hence, the first part of the proof follows by considering the unique positive solution $\mu$ of this system. The remaining statements easily follow by observing that $z(u)$ defines, as $\mu$ varies, a bundle of parabolas moving from left towards right as $\mu$ increases.\qquad
\end{proof}

Hereafter, we assume that the singularities $s_{j}^{\star} \in S^{\star}$ are arranged in ascending order with respect to the function $\varphi$, i.e.
\[
	0 = \varphi(s^{\star}_{0}) < \varphi(s^{\star}_{1}) < \dots < \varphi(s^{\star}_{J}) .
\]

We can therefore consider the $J+1$ parabolas defined by (\ref{eq:ParabolicContour}), each with $\mu=\varphi(s^{\star}_{j})$, and determining $J+1$ disjointed regions $R_j$ in the complex plane with the singularities $s^{\star}_{j}$ and $s^{\star}_{j+1}$ on the left and on the right boundary (except for the last region $R_{J}$ which is unbounded to the right). In Figure \ref{fig:Fig_PolesParabolas} we show, for instance, the complex plane partitioned into 6 regions of this kind (note that the first parabola, the one with  $\mu=\varphi(s^{\star}_{0})=0$, collapses onto the negative real axis).

\begin{figure}[ht]
	\centering
	\includegraphics[width=0.6\textwidth]{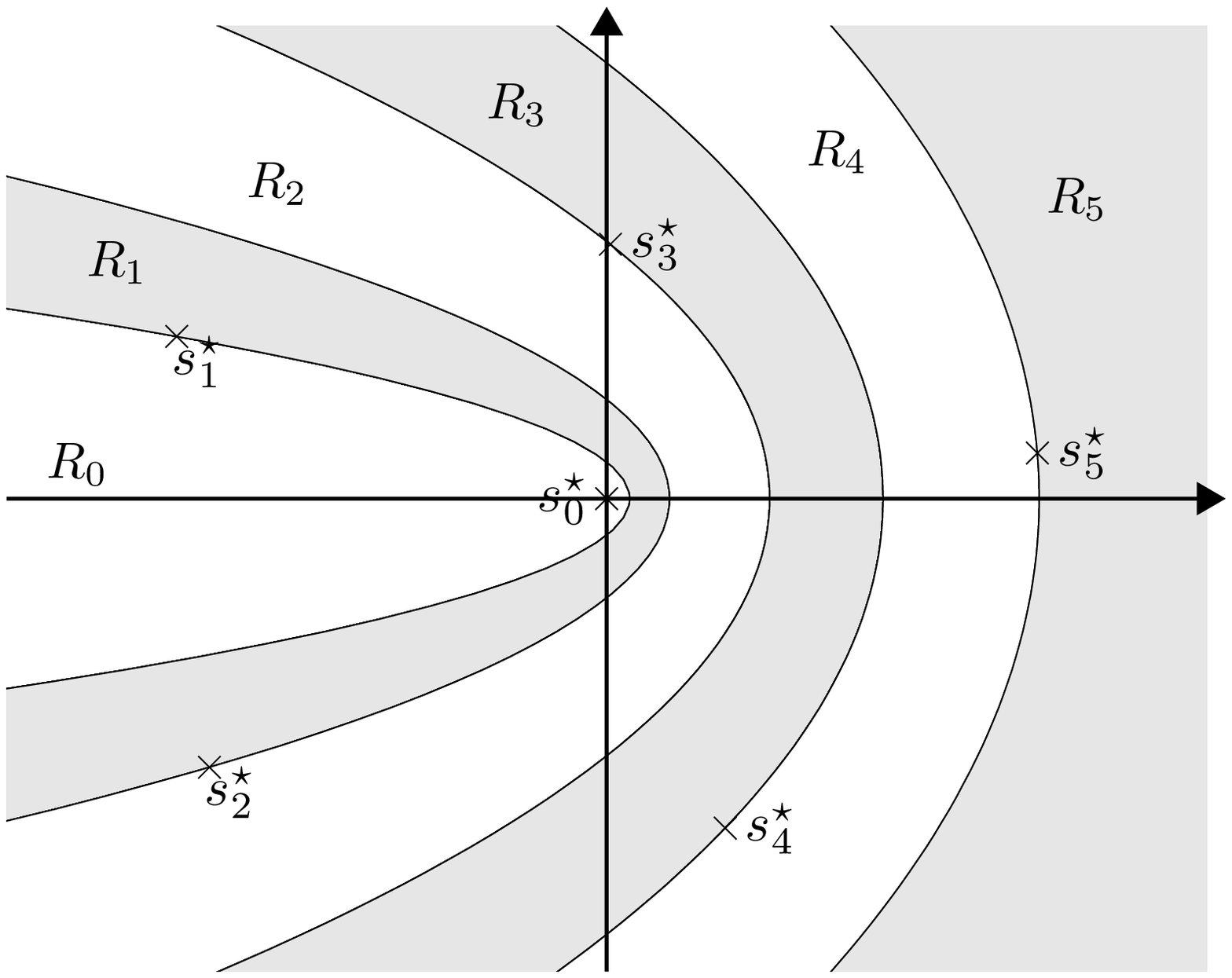} 
	\caption{Partitioning of $\Cset$ into some regions $R_{j}$ by means of parabolas (\ref{eq:ParabolicContour}) through $s^{\star}_j$.}
	\label{fig:Fig_PolesParabolas}
\end{figure}

A parabolic contour defined according to (\ref{eq:ParabolicContour}) is in a region $R_{j}$, $j \in \{0,1,\dots,J\}$, whenever $\mu$ satisfies 
\begin{equation}\label{eq:BoundsMuj}
	\varphi(s^{\star}_{j}) < \mu < \varphi(s^{\star}_{j+1}) .
\end{equation}
(for convenience, a fictitious singularity $s^{\star}_{J+1}$ with $\varphi(s^{\star}_{J+1})=+\infty$ is introduced). At the same time, if $z_j(u) = \mu_{j} (iu+1)^2$ is a parabolic contour with $\mu_{j}$ satisfying (\ref{eq:BoundsMuj}), then $R_{j}$ is the region of analyticity of Theorem \ref{thm:ErrorIntegrationContour} for $g_j(u) = e^{z_j(u)t} {\cal E}_{\alpha,\beta}(z_j(u);\lambda)z_j'(u)$ and can be expressed as 
\begin{equation}\label{eq:RegionConformalMap}
	R_{j} = \left\{ z \in \Cset \, | \, z = \mu_{j}(iw+1)^2 , \, w \in \Cset \, \textrm{ and } \,
	-d^{\star}_j<\Im(w)<c^{\star}_{j} \right\} ,
\end{equation}	
with
\begin{equation}\label{eq:BoundsCstarDstar}
	c^{\star}_{j} = \left\{ \begin{array}{ll}
		1 & j=0 \\
		\displaystyle
		1 - \sqrt{\frac{\varphi(s^{\star}_{j})}{\mu_{j}}} \quad & j \ge 1 \\
		\end{array} \right.
	\quad \quad
	d^{\star}_{j} = \left\{ \begin{array}{ll}
		\displaystyle
		\sqrt{\frac{\varphi(s^{\star}_{j+1})}{\mu_{j}}}-1 \quad & j \le J \\	
		+\infty & j = J  .\\
	\end{array} \right.
	\quad	
\end{equation}

Basically, by means of (\ref{eq:BoundsCstarDstar}) it is possible to represent each region of analyticity $R_{j}$ as the strip $-d^{\star}_j<\Im(w)<c^{\star}_{j}$ in the $w$-plane.

The estimation of the discretization error of Theorem \ref{thm:ErrorIntegrationContour} involves finding upper bounds for $M_{+}(c_{j})$ and $M_{-}(d_{j})$, with $c_{j} < c^{\star}_{j}$ and $d_{j} < d^{\star}_{j}$, in each region $R_{j}$. Matching the corresponding $DE_{+}(c_{j})$ and $DE_{-}(d_{j})$ with $TE = {\cal O} \bigl(\left|e^{z_{j}(hN)t}\right|\bigr)= {\cal O} \bigl(\left|e^{\mu_{j}(1-(h_{j}N_{j})^2t}\right|\bigr)$, according to a procedure similar to that devised in \cite{WeidemanTrefethen2007}, allows to obtain the optimal parameters (contour geometry $\mu_{j}$, step--size $h_{j}$ and number $N_{j}$ of quadrature nodes) in order to achieve a prescribed tolerance $\varepsilon>0$. 

The region involving the minimum computational effort (i.e., the one with the minimum number of quadrature nodes) is hence selected to perform the numerical inversion of the LT; the residues corresponding to the singularities left out by the selected contour are accordingly added in the final result as stated by (\ref{eq:GeneralizedMLInverseLaplace}). 

The main steps of the OPC algorithm can be therefore listed as follows: 
\begin{enumerate}
	\item estimation of $M_{+}(c_{j})$ and $M_{-}(d_{j})$ in each region $R_{j}$;
	\item matching, in each region $R_{j}$, of $DE_{+}(c_{j})$, $DE_{-}(c_{j})$ and $TE$ with the prescribed accuracy $\varepsilon >0$ and evaluation of the parameters $\mu_{j}$, $h_{j}$ and $N_{j}$;
	\item selection of the region $R_{j}$ in which to perform the integration on the basis of the lowest computation and reduction of round--off errors.
\end{enumerate}

\subsection{Estimation of $M_{+}(c_{j})$ and $M_{-}(d_{j})$ in each region $R_j$}\label{SS:EstimationMcMd}

To provide an estimation of $M_{+}(c_{j})$ we distinguish the case in which the region $R_{j}$ is bounded to the left by the singularity at the origin (i.e., $j=0$) and the case in which the singularity on the left boundary of $R_{j}$ is one of the poles of ${\cal E}_{\alpha,\beta}$ except the origin (i.e., $j>0$). 

To discuss the first case we introduce the following preliminary result.

\begin{lemma}\label{lem:AsymptoticIntegral}
Let $A,\sigma > 0$ and $p \in \Rset$. Then as $A \to 0$
\[
	\int_{-\infty}^{\infty} e^{-\sigma u^2} (u^2 + A)^{p} \, du = 
	\left\{ \begin{array}{ll}
		{\cal O} \bigl(1 \bigr) & \textrm{ if } p > -\frac{1}{2} \\
		{\cal O} \bigl( \log \sigma A \bigr)  & \textrm{ if } p = -\frac{1}{2} \\
		{\cal O} \bigl( A^{p+\frac{1}{2}} \bigr) & \textrm{ if } p < - \frac{1}{2}  .\\
	\end{array} \right.
\]
\end{lemma}

\begin{proof}
By splitting the integral into the two subintervals $(-\infty,0]$ and $[0,\infty)$ and making the change of variable $s=u^2/A$, it is possible to preliminarily observe that
\[
	\int_{-\infty}^{\infty} e^{- \sigma u^{2}} \left(u^{2}+A\right)^{p} du = 
	A^{p+\frac{1}{2}}  \int_{0}^{\infty} s^{-\frac{1}{2}} e^{- \sigma A s}   \bigl( s + 1\bigl)^{p} ds	.
\]

The right-hand side of the above equation is the integral representation of the confluent hypergeometric function of the second kind \cite{Tricomi1954}, namely the $\Psi(a,b,z)$ function with parameters $a=1/2$, $b=p+3/2$ and $z=\sigma A$, and hence
\[
	\int_{-\infty}^{\infty} e^{- \sigma u^{2}} \left(u^{2} + A\right)^{p} du
	= A^{p+\frac{1}{2}} \Gamma({\textstyle\frac{1}{2}}) \Psi \bigl( {\textstyle\frac{1}{2}}, p+{\textstyle\frac{3}{2}} , \sigma A \bigr) .
\]

As $z\rightarrow 0$ the $\Psi$--function admits the following asymptotic expansions \cite{WolframHypergeometricU}
\[
	\Psi(a,b,z) = \left\{ \begin{array}{ll}
		\displaystyle
		\sum_{j=0}^{\infty} u_{j}^{(1)} z^{j} +
		 z^{1-b} \sum_{j=0}^{\infty} v_{j}^{(1)} z^{j} 
		& \textrm{ if } b \not\in \Zset \\		
		\displaystyle
		\sum_{j=0}^{-b} u_{j}^{(2)} z^{j} + z^{1-b} \sum_{j=0}^{\infty} v_{j}^{(2)} z^{j} + \log(z) z^{1-b} \sum_{j=0}^{\infty} w_{j}^{(2)} z^{j}
		& \textrm{ if } b \in \Zset^{-} \cup \{0\} \\
		\displaystyle
		\sum_{j=0}^{\infty} u_{j}^{(3)} z^{j} + \sum_{j=1}^{b-1} v_{j}^{(3)} z^{-j}+ \log(z) \sum_{j=0}^{\infty} w_{j}^{(3)} z^{j} 
		& \textrm{ if } b \in \Zset^{+} \\	
	\end{array} \right.
\]
with coefficients $\{u_{j}^{(\ell)}\}_{j}$, $\{v_{j}^{(\ell)}\}_{j}$ and $\{w_{j}^{(\ell)}\}_{j}$, $\ell=1,2,3$, independent of $z$. The proof now follows by considering the leading terms in each summation. \qquad
\end{proof}

The first region $R_{0}$ is bounded to the left by the singularity $s^{\star}_{0}$ at the origin; since (\ref{eq:BoundsCstarDstar}), the corresponding upper bound for the strip of analyticity in the $w$-plane is $c^{\star}_{0}=1$. We can provide the following estimation for $M_{+}(c_{0})$ for $c_{0}<1$.

\begin{proposition}\label{prop:Mcpuls1}
Let $\mu_{0}$ be such that $0<\mu_{0}<\varphi(s^{\star}_{1})$. For any $c_{0}< 1$ there exists a constant $\bar{M}_{+}$ (independent of $c_{0}$)  such that 
\[
	M_{+}(c_{0}) < \bar{M}_{+} \cdot \tilde{M}(c_{0}) , 
\]
where as $c\to 1$ it is
\[
	\tilde{M}(c) = \left\{ \begin{array}{ll}
		{\cal O} \bigl(1 \bigr) & \textrm{ if } \beta < \alpha + 1  \\
		{\cal O} \bigl( \log \mu_{0} t (1-c)^2 \bigr) = {\cal O} \bigl( \log (1-c) \bigr) & \textrm{ if } \beta = \alpha + 1  \\
		{\cal O} \bigl( (1-c)^{2(\alpha-\beta+1)} \bigr) & \textrm{ if } \beta > \alpha + 1 . \\
	\end{array} \right.
\]
\end{proposition}

\begin{proof}
By replacing (\ref{eq:ParabolicContour}) and $z'(u) = 2 \mu_{0} (i-u)$ in (\ref{eq:IntegrandFunction}), we preliminarily obtain
\[
	g(u+ir) = 2 i \mu_{0}^{\alpha-\beta+1} e^{\mu_{0} ((1-r) +iu)^2 t} \frac{ ((1-r) +iu)^{2(\alpha-\beta)+1}}{z(u+ir)^{\alpha}-\lambda} .
\]

Since $\lambda \not=0$ (see Remark \ref{rmk:lambdanot0}), it is natural to assume the existence of a positive $\hat{M}$ such that $\left| z(u + ir)^{\alpha}-\lambda \right| \ge \hat{M}$ for any $r \in [0,c^{\star}_{0})$. Hence
\[
	\bigl| g(u + ir) \bigr| 
		\le 2 \mu^{\alpha-\beta+1} e^{\mu_{0} (1 - r)^2 t} \hat{M} 
		e^{-\mu_{0} u^2 t} \left(u^2 + (1 - r)^2\right)^{\alpha-\beta+1/2}
\]
and, after putting for shortness $\bar{M}_{+} = 2 \mu_{0}^{\alpha-\beta+1} e^{\mu_{0} t} \hat{M}$, we have for any $c_{0} < 1$
\[
	M_{+}(c_{0}) \le \bar{M}_{+} \cdot \tilde{M}(c_{0})
	,  \quad
	\tilde{M}(c_{0}) 
	= \max_{0\le r \le c_{0}} \int_{-\infty}^{\infty} e^{-\mu_{0} u^2 t} \left(u^2 + (1 - r)^2\right)^{\alpha-\beta+1/2} \, du	.
\]

The proof now follows after applying Lemma \ref{lem:AsymptoticIntegral}. \qquad
\end{proof}

With $\lambda$ very close to $0$, in the above proof it is possible that $\hat{M}\ll 1$, thus affecting the asymptotic estimation for $\tilde{M}(c)$. In this case, we are in the presence of a very narrow region $R_{0}$ which, as we will discuss later in the final part of subsection \ref{SS:QuadratureBoundedRight}, must be discarded since it does not allow one to achieve the assigned tolerance. For this reason, we do not consider the effects on $\tilde{M}(c)$ of a possibly very small  $\hat{M}$.

We now consider the regions $R_{j}$, $j=1,\dots,J$, which are bounded to the left by one of the poles $s^{\star}_j$ of ${\cal E}_{\alpha,\beta}(s;\lambda)$ except the origin (i.e., $\varphi(s^{\star}_j) > 0$).

\begin{lemma}\label{lem:Maggiorazioni}
Let $a,b\in \Rset$, with $b>a>0$. For $0 < \alpha \le 1$ it is 	$b^{\alpha} - a^{\alpha} \ge \alpha b^{\alpha-1} (b-a)$ and for $\alpha > 1$ it is $b^{\alpha} - a^{\alpha} \ge \alpha a^{\alpha-1} (b-a)$.
\end{lemma}
\begin{proof}
For $0 < \alpha \le 1$ it is immediate to verify that
\[
	b^{\alpha} - a^{\alpha} = \alpha \int_{a}^{b} s^{\alpha-1} \, ds \ge \alpha \int_{a}^{b} b^{\alpha-1} \, ds 
	= \alpha b^{\alpha-1} (b-a)
\]
and in a similar way the proof follows for $\alpha > 1$. \qquad
\end{proof}

\begin{proposition}\label{prop:Mcpuls2}
Let $j\in\{1,\dots,J\}$, $\mu_{j}>0$ such that $\varphi(s^{\star}_j)<\mu_{j}<\varphi(s^{\star}_{j+1})$  and $c_{j}^{\star}$ the upper bound of the strip of analyticity (\ref{eq:RegionConformalMap}) corresponding to $R_{j}$. For any $c_{j} < c_{j}^{\star}$ there exists $\bar{M}_{+}>0$ (independent of  $c_{j}$) such that 
\[
	M_{+}(c_{j}) < \bar{M}_{+} \cdot ( c_{j}^{\star} - c_{j})^{-1} .
\]
\end{proposition}

\begin{proof}
Let $c_{j} < c_{j}^{\star}$ and consider $r \in [0,c_{j}]$. By Proposition \ref{prop:PositionPointParabola}, $s_{j}^{\star}$ lies on the parabola $z^{\star}_{j}(u) = \varphi(s^{\star}_j)(iu+1)^2$ and hence $\lambda=(s_{j}^{\star})^{\alpha}$ lies on the curve $\bigl(z^{\star}_{j}(u)\bigr)^{\alpha} = \bigl(\varphi(s^{\star}_j)\bigr)^{\alpha}(iu+1)^{2\alpha}$. The distance between $(z_{j}(u+ir))^{\alpha}$ and $\lambda$ is therefore greater than the distance between $(z_{j}(u+ir))^{\alpha}$ and $\bigl(z^{\star}_{j}(u)\bigr)^{\alpha}$ evaluated at $u=0$, i.e. 
\[
	\bigl|(z_{j}(u+ir))^{\alpha}-\lambda\bigr| 
	\ge \left| \mu_{j}^{\alpha} \bigl(1-r\bigr)^{2\alpha} - \bigr(\varphi(s^{\star}_{j})\bigl)^{\alpha} \right|  .
\]

Since $\varphi(s^{\star}_{j}) < \mu_{j} \bigl(1-r\bigr)^2$, Lemma \ref{lem:Maggiorazioni} yields
\[
	\bigl|(z_{j}(u+ir))^{\alpha}-\lambda\bigr| 
	\ge \alpha P_{\alpha,j} \left( \mu_{j} \bigl(1-r\bigr)^2 - \varphi(s^{\star}_{j}) \right) ,
\]
where $P_{\alpha,j}=\bigr(\varphi(s^{\star}_{j})\bigl)^{\alpha-1}$ when $0 < \alpha < 1$ and $P_{\alpha,j} = \mu_{j}^{\alpha-1} \bigl(1-r\bigr)^{2\alpha-2}$ when $\alpha > 1$. It is elementary to see that for any $r \le c_{j} < c_{j}^{\star}$ it is
\begin{eqnarray*}
	\mu_{j} \bigl(1-r\bigr)^2 - \varphi(s^{\star}_{j}) 
	&=& \mu_{j} \left( \bigl(1-r\bigr)^2 - \frac{\varphi(s^{\star}_{j})}{\mu_{j}} \right) 
	 = \mu_{j} \left( \bigl(1-r\bigr)^2 - (1-c_{j}^{\star})^2 \right)  \\
	&=& \mu_{j} ( c_{j}^{\star} - r) (2- r - c_{j}^{\star}) 
	 >  2 \mu_{j} ( c_{j}^{\star} - c_{j}) (1 - c_{j}^{\star}) \
\end{eqnarray*}

Since form (\ref{eq:BoundsCstarDstar}) $\varphi(s^{\star}_{j})/\mu_{j} = (1-c_{j}^{\star})^2$, when $0<\alpha<1$ we can easily verify that
\[
	\bigl|(z_{j}(u+ir))^{\alpha}-\lambda\bigr| 
	> 2 \alpha \bigr(\varphi(s^{\star}_{j})\bigl)^{\alpha-1} \mu_{j} ( c_{j}^{\star} - c_{j}) (1 - c_{j}^{\star})
	= 2 \alpha \bigr(\varphi(s^{\star}_{j})\bigl)^{\alpha} ( c_{j}^{\star} - c_{j}) (1 - c_{j}^{\star})^{-1}, 
\]
while, for $\alpha > 1$, it is instead
\[
	\bigl|z_{j}(u+ir)^{\alpha}-\lambda\bigr| 
	> 2 \alpha \mu_{j}^{\alpha} ( c_{j}^{\star} - c_{j}) (1 - c_{j}^{\star})^{2\alpha-1} .
\]

We denote with $Q_{\alpha,j}$ the constant independent of $c_{j}$
\[
	Q_{\alpha,j} = \left\{ \begin{array}{ll}
		2 \alpha \bigr(\varphi(s^{\star}_{j})\bigl)^{\alpha} (1 - c_{j}^{\star})^{-1} & \textrm{if} \, 0 < \alpha < 1 \\
		2 \alpha \mu_{j}^{\alpha}  (1 - c_{j}^{\star})^{2\alpha-1} \quad & \textrm{if} \, \alpha > 1 \
	\end{array}\right.
\]
and write the inequality
\[
	\bigl|(z_{j}(u+ir))^{\alpha}-\lambda\bigr| > Q_{\alpha,j} ( c_{j}^{\star} - c_{j})
\]
which allows to obtain
\[
	M_{+}(c_{j}) = \max_{0\le r\le c_{j}} \int_{-\infty}^{\infty} |g(u+ir)|du 
	< \bar{M}_{+} ( c_{j}^{\star} - c_{j})^{-1}, 
\]
with $\bar{M}_{+} = 2 \mu_{j}^{\alpha-\beta+1} e^{\mu_{j} t} \hat{M}_{+}$ and
\[
	\hat{M}_{+} = \frac{1}{Q_{\alpha,j}}  \max_{0 \le r \le c_{j}^{\star}} \hat{M}(r)
	, \quad
	\hat{M}(r) = \int_{-\infty}^{\infty} e^{-\mu u^2 t} \left(u^2 + (1 - r)^2\right)^{\alpha-\beta+1/2} \, du
\]
being $\hat{M}(r) < + \infty$ since $r \le c_{j}^{\star} < 1$. \qquad
\end{proof}

Also for $M_{-}(d_{j})$ it is necessary to distinguish two cases: when the computation is performed in a region $R_{j}$, $j=0,\dots,J-1$, bounded to the right and when the integration is instead performed in last region $R_{J}$. 

\begin{proposition}\label{prop:Mdpuls}
Let $j\in\{0,\dots,J-1\}$, $\mu_{j}>0$ such that $\varphi(s^{\star}_j)<\mu_{j}<\varphi(s^{\star}_{j+1})$  and $d_{j}^{\star}$ the lower bound of the strip of analyticity (\ref{eq:RegionConformalMap}) corresponding to $R_{j}$. For any $d_{j} < d_{j}^{\star}$ there exists $\bar{M}_{-}>0$ (independent of  $d_{j}$) such that 
\[
	M_{-}(d_{j}) < \bar{M}_{-} \cdot ( d_{j}^{\star} - d_{j})^{-1} .
\]
\end{proposition}

\begin{proof}
The prof is symmetric to the proof of Proposition \ref{prop:Mcpuls2} and we omit the details. We just point out that the term $\bar{M}_{-}$ is now given by $\bar{M}_{-} = 2 \mu_{j}^{\alpha-\beta+1} e^{\mu_{j} (1+d_{j}^{\star})^2t} \hat{M}_{-}$ with $\hat{M}_{-}$ obtained in a similar way as $\hat{M}_{+}$.
\end{proof}


The discussion for $M_{-}(d_{J})$ in the last (right--unbounded) region $R_{J}$ is the same as proposed in \cite{WeidemanTrefethen2007}. We just recall that an upper bound for $M_{-}(d_{J})$ is achieved at $d_{J}=\pi/(\mu_{J} t h_{J}) -1$ which allows to write 
\begin{equation}\label{eq:DE_LastRegion}
	DE_{-}(d_{J}) = {\cal O}\bigl(e^{\pi^2/(\mu_{J} t h_{J}^2) + 2\pi/h_{J}} \bigr), \quad h_{J} \to 0 .
\end{equation}

\subsection{Evaluation of the quadrature parameters}\label{SSS:BalancigX1unbounded}

Thanks to the analysis carried out in Subsection \ref{SS:EstimationMcMd}, and after highlighting the exponential growing factor in $\bar{M}_{-}$, some upper bounds for the discretization errors are now available in the form
\[
	DE_{+}(c_{j}) \le \bar{M}_{+} (c_{j}^{\star} - c_{j})^{-p_{j}} e^{-2\pi c_{j}^{\star} /h_{j}}
	, \quad
	DE_{-}(d_{j}) \le \bar{M}_{-} (d_{j}^{\star} - d_{j})^{-q_{j}} e^{-2\pi d_{j}^{\star} /h_{j} + \varphi(s^{\star}_{j+1}) t}
\]
for $c_{j}<c_{j}^{\star}$, $d_{j}<d_{j}^{\star}$ and some nonnegative values $p_{j}$ and $q_{j}$; only in the last region $R_{J}$ a different result applies for $DE_{-}(d_{j})$, according to (\ref{eq:DE_LastRegion}). 

Unless $p_{j}=0$ and $q_{j}=0$, the presence of the algebraic terms $(c_{j}^{\star} - c_{j})^{-p_{j}}$ and $(d_{j}^{\star} - d_{j})^{-q_{j}}$ cannnot be disregarded; they can be indeed very large and an unfit selection of $c_{j}$ and $d_{j}$ can lead to an incorrect error analysis as already observed in the first part of \cite{GarrappaPopolizio2013}.

The task of including, in the error analysis, the contribution of a possibly large algebraic term was accomplished in \cite{GarrappaPopolizio2013} by introducing an auxiliary variable and expressing the parameters in the formula for the numerical inversion of the LT in terms of this variable; the optimal value of the auxiliary variable was hence selected by minimizing the number of quadrature nodes in order to keep the computational effort at a minimum. Numerical experiments showed that despite the nonnegligible computation required by finding the minimum of a nonlinear function, the overall computation was the same performed in an efficient way. 

The work in \cite{GarrappaPopolizio2013} was anyway devoted to the evaluation of the ML function (\ref{eq:GeneralizedML}) on the real negative semiaxis and, mainly, for $0<\alpha<1$, thus involving just one singularity, namely at origin. In the more general context this approach is no more feasible: since most of the regions $R_{j}$ have two distinct singularities, the introduction of two auxiliary variables leads to the need of finding the minimum of a nonlinear function with two variables, a problem whose solution can be quite expensive.

We propose here a completely different approach to take into account the algebraic terms in $DE_{+}(c_{j})$ and $DE_{-}(d_{j})$. To this purpose we distinguish again two main cases: the case of a bounded to the right region $R_{j}$ (i.e., $j=0,\dots,J-1$) and the case of the right--unbounded region (i.e., the last region $R_{J}$).

\subsubsection{Quadrature parameters in a region bounded to the right}\label{SS:QuadratureBoundedRight}

The most straightforward way to prevent the possible growth of the terms $(c^{\star}_j - c_{j})^{-p_{j}}$ and $(d^{\star}_j - d_{j})^{-q_{j}}$ in the discretization errors $DE_{+}(c_j)$ and $DE_{-}(d_j)$ is by forcing $\mu_{j}$ to belong to a subinterval $[\bar{\varphi}_{j}^{\star} , \bar{\varphi}_{j+1}^{\star}]$ instead of $[\varphi(s_{j}^{\star}) , \varphi(s_{j+1}^{\star})]$, with 
\begin{equation}\label{eq:EssentialConstraint}
	\varphi(s_{j}^{\star}) \le \bar{\varphi}_{j}^{\star} < \bar{\varphi}_{j+1}^{\star} < \varphi(s_{j+1}^{\star}) 
\end{equation}
(the equality in (\ref{eq:EssentialConstraint}) is introduced just to cover the case in which $p_{j}=0$). Under the conformal map $z_{j}(w) = \mu_{j}(iw+1)^2$, the strip $-\bar{d}_{j}^{\star} < \Im(w) < \bar{c}_{j}^{\star}$ corresponding to the region determined by the parabolas through $\bar{\varphi}_{j}^{\star}$ and $\bar{\varphi}_{j+1}^{\star}$ is given by
\begin{equation}\label{eq:BoundsReducedRegion}
	\bar{c}_{j}^{\star} = 1 - \sqrt{\frac{\bar{\varphi}_{j}^{\star}}{\mu_{j}}} < c_{j}^{\star}
	, \quad
	\bar{d}_{j}^{\star} = \sqrt{\frac{\bar{\varphi}_{j+1}^{\star}}{\mu_{j}}} - 1 < d_{j}^{\star}.
\end{equation}

A suitable strategy is to fix an arbitrary value $\bar{f}>1$ and select $\bar{\varphi}_{j}^{\star}$ and $\bar{\varphi}_{j+1}^{\star}$ by forcing
\begin{equation}\label{eq:ForceFbar}
	(c^{\star}_j - \bar{c}_{j}^{\star})^{-p_{j}} = (d^{\star}_j - \bar{d}_{j}^{\star})^{-q_{j}} = \bar{f}
\end{equation}

Whenever (\ref{eq:ForceFbar}) can be imposed without a large value of $\bar{f}$, for instance $\bar{f}\approx1$, the terms $(c^{\star}_j - \bar{c}_{j}^{\star})^{-p_{j}}$ and $(d^{\star}_j - \bar{d}_{j}^{\star})^{-q_{j}}$ can be neglected in $DE_{+}(c_{j})$ and $DE_{-}(d_{j})$. Within very narrow regions $R_{j}$ it cannot be possible to satisfy (\ref{eq:ForceFbar}) for sufficiently small values of $\bar{f}$; the accuracy $\varepsilon$ must be therefore scaled as $\bar{\varepsilon}=\varepsilon/\bar{f}$ before removing $(c^{\star}_j - \bar{c}_{j}^{\star})^{-p_{j}}$ and $(d^{\star}_j - \bar{d}_{j}^{\star})^{-q_{j}}$ from the discretization errors. Obviously, this scaling also has an effect on the truncation error $TE$ and leads to slightly increase the number $N_j$ of quadrature nodes.

The asymptotic balancing of the different components of the error now reads 
\begin{equation}\label{eq:Balancing1}
	-\frac{2 \pi \bar{c}_{j}^{\star}}{h_{j}} = -\frac{2 \pi \bar{d}_{j}^{\star}}{h_{j}} + \bar{\varphi}_{j+1}^{\star} t
	= \mu_{j} t \bigl(1 - h_{j}^2 N_{j}^2 \bigr) = \log \bar{\varepsilon}
\end{equation}
from which it is immediate to obtain 
\begin{equation}\label{eq:muj_BoundedRegion}
	\mu_{j} = \left( \frac{(1+w) \sqrt{\bar{\varphi}_{j}^{\star}} + \sqrt{\bar{\varphi}_{j+1}^{\star}}}{2+w} \right)^2
	, \quad
	w = - \frac{\bar{\varphi}_{j+1}^{\star} t}{\log \bar{\varepsilon}} ,
\end{equation}
and
\[
	h_{j} 
	= -\frac{2\pi}{\log \bar{\varepsilon}} \cdot \frac{\sqrt{\bar{\varphi}_{j+1}^{\star}} - \sqrt{\bar{\varphi}_{j}^{\star}}}{(1+w)\sqrt{\bar{\varphi}_{j}^{\star}} + \sqrt{\bar{\varphi}_{j+1}^{\star}}} 
	, \quad
	N_{j} = \frac{1}{h_{j}} \sqrt{1 - \frac{\log \bar{\varepsilon}}{t \mu_{j} }} .
\]

An essential task is to select $\bar{f}$ small enough to make $(c^{\star}_j - \bar{c}_{j}^{\star} )^{-p_{j}}$ and $(d^{\star}_j - \bar{d}_{j}^{\star} )^{-q_{j}}$ negligible in $DE_{+}(c_{j})$ and $DE_{-}(d_{j})$ and, at the same time, satisfy (\ref{eq:EssentialConstraint}). To this purpose we explicitly represent $c^{\star}_j - \bar{c}_{j}^{\star}$ and $d^{\star}_j - \bar{d}_{j}^{\star}$ in terms of $\bar{\varphi}_{j}^{\star}$ and $\bar{\varphi}_{j+1}^{\star}$ as

\[
	c^{\star}_j - \bar{c}_{j}^{\star} 
	=  \frac{ (2+w) \left( \sqrt{\bar{\varphi}_{j}^{\star}} - \sqrt{\varphi(s_{j}^{\star})}\right) }{(1+w)\sqrt{\bar{\varphi}_{j}^{\star}} + \sqrt{\bar{\varphi}_{j+1}^{\star}}}
	, \quad
	d^{\star}_j - \bar{d}_{j}^{\star} 
	= \frac{ (2+w) \left( \sqrt{\varphi(s_{j+1}^{\star})} - \sqrt{\bar{\varphi}_{j+1}^{\star}} \right)}{(1+w)\sqrt{\bar{\varphi}_{j}^{\star}} + \sqrt{\bar{\varphi}_{j+1}^{\star}}} .
\]

The obvious assumption $\bar{f}>1$ is sufficient to ensure that $\varphi(s_{j}^{\star}) < \bar{\varphi}_{j}^{\star}$ and $\bar{\varphi}_{j+1}^{\star} < \varphi(s_{j+1}^{\star})$; anyway, a minimum threshold value $\bar{f}_{min}>1$ must be determined with the aim of fulfilling $\bar{\varphi}_{j}^{\star}< \bar{\varphi}_{j+1}^{\star}$ for $\bar{f}>\bar{f}_{min}$.

When $p_{j} = 0$, a simple computation allows to verify that $\varphi(s_{j}^{\star})=\bar{\varphi}_{j}^{\star}<\bar{\varphi}_{j+1}^{\star}$ for
\[
	\bar{f}_{min} = \left( \frac{\sqrt{\varphi(s_{j}^{\star})}}{\sqrt{\varphi(s_{j+1}^{\star})}-\sqrt{\varphi(s_{j}^{\star})}}\right)^{q_{j}} ; 
\]
in the more general case we provide the following result (note that in regions bounded to the right it is always $q_{j}\not=0$).


\begin{proposition}
Let $p_{j},q_{j} > 0$ and $r_{j} = \max\{p_{j},q_{j}\}$. If 
\begin{equation}\label{eq:Boundfbar}
	\bar{f}>\bar{f}_{min} , \quad
	\bar{f}_{min} = \left(\frac{\sqrt{\varphi(s_{j}^{\star})}+\sqrt{\varphi(s_{j+1}^{\star})}}
	{\sqrt{\varphi(s_{j+1}^{\star})}-\sqrt{\varphi(s_{j}^{\star})}}\right)^{r_{j}} ,
\end{equation}
then $\bar{\varphi}_{j}^{\star}< \bar{\varphi}_{j+1}^{\star}$.
\end{proposition}

\begin{proof}
It is elementary to verify that $(c^{\star}_j - \bar{c}_{j}^{\star} )^{-p_{j}} = (d^{\star}_j - \bar{d}_{j}^{\star} )^{-q_{j}} = \bar{f}$ when $\bar{\varphi}_{j}^{\star}$ and $\bar{\varphi}_{j+1}^{\star}$ are obtained after solving the linear system
\[
	\left( \begin{array}{cc}
		2+w - (1+w)\bar{f}^{-1/p_{j}} & -\bar{f}^{-1/p_{j}} \\ \\
		(1+w)\bar{f}^{-1/q_{j}} & 2+w+\bar{f}^{-1/q_{j}} \\
	\end{array} \right)
	\left( \begin{array}{c}
		\sqrt{\bar{\varphi}_{j}^{\star}} \\ \\ \sqrt{\bar{\varphi}_{j+1}^{\star}} \\
	\end{array} \right)
	= (2+w) 	\left( \begin{array}{c}
		\sqrt{\varphi(s_{j}^{\star})} \\ \\ \sqrt{\varphi(s_{j+1}^{\star})}
	\end{array} \right)
\]
whose solutions can be explicitly formulated as 
\[
	\bar{\varphi}_{j}^{\star} = \left(\frac{(2+w+\bar{f}^{-1/q_{j}})\sqrt{\varphi(s_{j}^{\star})} + \bar{f}^{-1/p_{j}}\sqrt{\varphi(s_{j+1}^{\star})}}{2+w-(1+w)\bar{f}^{-1/p_{j}}+\bar{f}^{-1/q_{j}}}\right)^2
\]
and
\[
	\bar{\varphi}_{j+1}^{\star} = \left(\frac{ -(1+w)\bar{f}^{-1/q_{j}} \sqrt{\varphi(s_{j}^{\star})} + (2+w-(1+w)\bar{f}^{-1/p_{j}})\sqrt{\varphi(s_{j+1}^{\star})}}{2+w-(1+w)\bar{f}^{-1/p_{j}}+\bar{f}^{-1/q_{j}}}\right)^2 .
\]

By the hypothesis (\ref{eq:Boundfbar}) it is
\[
	\bar{f}^{-1/p_{j}} < \frac{\sqrt{\varphi(s_{j}^{\star})}+\sqrt{\varphi(s_{j+1}^{\star})}}
	{\sqrt{\varphi(s_{j+1}^{\star})}-\sqrt{\varphi(s_{j}^{\star})}} < 1
	, \quad
	\bar{f}^{-1/q_{j}} < \frac{\sqrt{\varphi(s_{j}^{\star})}+\sqrt{\varphi(s_{j+1}^{\star})}}
	{\sqrt{\varphi(s_{j+1}^{\star})}-\sqrt{\varphi(s_{j}^{\star})}} < 1
\]
and hence 
\begin{equation}\label{eq:tempBound}
	\bar{f}^{-1/q_{j}} \sqrt{\varphi(s_{j}^{\star})} + \bar{f}^{-1/p_{j}} \sqrt{\varphi(s_{j+1}^{\star})}
	< \sqrt{\varphi(s_{j+1}^{\star})} - \sqrt{\varphi(s_{j+1}^{\star})} .
\end{equation}

A simple computation allows to prove that
\[
	\bar{\varphi}_{j}^{\star} 
	= \bar{\varphi}_{j+1}^{\star} + (2+w)
		\frac{ \sqrt{\varphi(s_{j+1}^{\star})} - \sqrt{\varphi(s_{j+1}^{\star})} + \bar{f}^{-1/q_{j}} \sqrt{\varphi(s_{j}^{\star})} + \bar{f}^{-1/p_{j}} \sqrt{\varphi(s_{j+1}^{\star})}}
		     {2+w-(1+w)\bar{f}^{-1/p_{j}}+\bar{f}^{-1/q_{j}}}
\]
from which the proof follows after using (\ref{eq:tempBound}). \qquad
\end{proof}

In very narrow regions the value of $\varphi(s_{j}^{\star})$ can be very close to $\varphi(s_{j+1}^{\star})$ and the threshold $\bar{f}_{min}$ can be too large to assure the achievement of a small tolerance $\varepsilon>0$; in this case no contours can be selected and the region must be discarded.

\begin{remark}
A more conservative error analysis would take into account also the exponential growing term  $e^{\mu_j t}$ in $\bar{M}_{+}$ (see the proof of Propositions \ref{prop:Mcpuls1} and \ref{prop:Mcpuls2}). In this case, and by using for simplicity the upper bound $e^{\bar{\varphi}_{j+1}^{\star} t}$, the integration parameters are obtained after replacing $\log \bar{\varepsilon}$ with $\log \bar{\varepsilon}-\bar{\varphi}_{j+1}^{\star} t$  in the formulas for $w$ and $h_{j}$. This change however does not seem to offer substantial improvements since it actually exerts the effects in regions with large $\varphi(s_{j}^{\star})$ and $\varphi(s_{j+1}^{\star})$ which are normally discarded for accuracy reasons as we will discuss later on. 
\end{remark}

\subsubsection{Quadrature parameters in an unbounded region to the right}\label{SSS:BalancigX1unbounded_OLD}

In the last and right--unbounded region $R_{J}$ the balancing of the exponential factors of the errors leads to
\begin{equation}\label{eq:BalanceUnbounded}
	-\frac{2 \pi \bar{c}^{\star}_{J}}{h_{J}}  
	= - \frac{\pi^{2}}{\mu_{J}t h_{J}^2} + \frac{2\pi}{h_{J}}
	= \mu_{J} t \bigl(1 - h_{J}^2 N_{J}^2 \bigr) = \log \varepsilon ,
\end{equation}
where $\bar{c}^{\star}_{J}  < c^{\star}_{J}$ is selected according to (\ref{eq:BoundsReducedRegion}) 
for $\bar{\varphi}_{J}^{\star} > \varphi(s_{J}^{\star})$, from which we obtain 
\begin{equation}\label{eq:param_rapr}
	h_{J} = \frac{1 + 2\bar{c}^{\star}_{J}}{N_{J}}
	, \quad
	\mu_{J} = \frac{\pi N_{J}}{2 t (1+\bar{c}^{\star}_{J}) (1 + 2\bar{c}^{\star}_{J})} 
	, \quad
	N_{J} = - \frac{1+2\bar{c}^{\star}_{J}}{2\pi \bar{c}^{\star}_{J}} \log \varepsilon .
\end{equation}

Unfortunately, because of the implicit dependence on the unknown $\mu_{J}$, we cannot use (\ref{eq:BoundsReducedRegion}) to determine $\bar{c}^{\star}_{J}$. It is therefore necessary to formulate $h_{J}$, $\mu_{J}$ and $N_{J}$ directly in terms of $\bar{\varphi}_{J}^{\star}$ instead of $\bar{c}^{\star}_{J}$. Since it is
\begin{equation}\label{eq:mu_rapr}
	\mu_{J} = \frac{\bar{\varphi}_{J}^{\star}}{(1- \bar{c}^{\star}_{J})^2} ,
\end{equation}
by matching the two representations of $\mu_{J}$ in (\ref{eq:param_rapr}) and (\ref{eq:mu_rapr}) we obtain a second order algebraic equation with respect to $\bar{c}^{\star}_{J}$ whose unique solution satisfying $\bar{c}^{\star}_{J}<1$ is
\[
	\bar{c}^{\star}_{J} = \frac{3 + A - \sqrt{1 + 12 A}}{ A- 4}
	, \quad
	A = \frac{\pi N_{J}}{t \bar{\varphi}_{J}^{\star} }  .
\]

A straightforward manipulation leads to 
\[
	h_{J} = \frac{1}{N_{J}} \left( - \frac{3A}{4-A} - \frac{2 - 2 \sqrt{1 + 12A}}{4-A} \right)
	, \quad
	\mu_{J} = \frac{\bar{\varphi}_{J}^{\star} (4-A)^2}{(7 - \sqrt{1+12A})^2}
\]
and, after imposing $-{2 \pi \bar{c}^{\star}_{J}}/{h_{J}} = \log \varepsilon$, we obtain 
\[
	N_{J} 
	= - \frac{1+2\bar{c}^{\star}_{J}}{2\pi \bar{c}^{\star}_{J}} \log \varepsilon
	= \frac{\bar{\varphi}_{J}^{\star} t}{\pi} \left( 1 - \frac{3 \log \varepsilon}{2 \bar{\varphi}_{J}^{\star} t} + \sqrt{ 1 - 2 \frac{ \log \varepsilon}{ t \bar{\varphi}_{J}^{\star} }} \right) .
\]

A direct evaluation of a suitable value for $\bar{\varphi}_{J}^{\star}$ is now not possible since its implicit dependence on $\mu_{J}$. We hence use an iterative process by which, starting from an initial guess very close to $\varphi(s_{J}^{\star})$, the value of $\bar{\varphi}_{J}^{\star}$ is increased until the corresponding value of $\bar{f}$, evaluated as
\begin{equation}\label{eq:Extimation_fbar}
	\bar{f} = (c^{\star}_{J} - \bar{c}^{\star}_{J})^{-p_{J}} 
	= \left( \frac{\sqrt{\bar{\varphi}_{J}^{\star}} - \sqrt{\varphi(s_{J}^{\star})}}{\sqrt{\mu_{J}}} \right)^{-p_{J}} ,
\end{equation}
does not fall into an interval $[\bar{f}_{min}, \bar{f}_{max}]$ which is a priori selected, for instance $[1, 10]$. To this aim a target value $\bar{f}_{tar} \in [\bar{f}_{min}, \bar{f}_{max}]$  can be established and  the new attempted value for $\bar{\varphi}_{J}^{\star}$ is obtained after replacing $\bar{f}$ with $\bar{f}_{tar}$ in (\ref{eq:Extimation_fbar}). In our experiments we have observed the convergence of this procedure in very few (usually 1 or 2) iterations.

\subsection{Selection of the region in which to invert the LT}

After evaluating parameters $\mu_{j}$, $h_j$ and $N_j$ in each subregion $R_j$, we select the region involving the minimum number $N_j$ of quadrature nodes to actually perform the numerical inversion of the LT with the minimum computational effort.




Because of the presence of the factor $e^{\mu t}$, with large values of $t$ and/or $\mu$ it is possible the presence in the summation $I_{h,N}$ of terms with large magnitude and terms with small magnitude; the effects of this simultaneous presence are in numerical cancellation which can become catastrophic. 

As already observed in \cite{Weideman2010}, the rounding error is roughly $RE \approx e^{\mu t}\epsilon$, with $\epsilon$ the machine precision. To keep rounding errors below the desired accuracy $\varepsilon > \epsilon$ it is therefore necessary that $\mu_{j} < (\log \varepsilon - \log \epsilon)/t$ and, hence, from (\ref{eq:muj_BoundedRegion}) it is sufficient to verify 

\[
	\sqrt{\bar{\varphi}_{j}^{\star}} + \sqrt{\bar{\varphi}_{j+1}^{\star}} 
	< 2 \sqrt{ (\log \varepsilon - \log \epsilon)/t} .
\]

In regions with $\varphi(s_{j}^{\star}) > (\log \varepsilon - \log \epsilon)/t$ this condition cannot be fulfilled; in order prevent round--off errors from destroying all of the significant digits, such regions must be discarded and the computation moved to one of the remaining regions. In the other cases the above equation provides a bound for $\bar{\varphi}_{j+1}^{\star}$. 


Another possible source for numerical cancellation is the closeness of the contour to one of the singularities on the boundary of the region $R_{j}$. We observe however that, despite the previous case in which the accuracy is affected by a factor proportional to $e^{\mu t}\epsilon$, in this case the accuracy is affected only in an algebraic way and, as observed by means of numerical experiments, it is sufficient to select $\bar{\varphi}_{j}^{\star}$ and $\bar{\varphi}_{j+1}^{\star}$ as previously described in order to avoid the cancellation.

In the last region $R_{J}$ it is possible, even when $\varphi(s_{J}^{\star}) > (\log \varepsilon - \log \epsilon)/t$, that the value $\mu_{J}$ resulting from the balancing of the various components of the error is too large and the round-off error $RE \approx e^{\mu_{J}t} \epsilon$ exceeds the required tolerance $\varepsilon$. Since in this case $RE$ dominates the discretization error $DE_{-}$ \cite{Weideman2010}, it is necessary to replace in (\ref{eq:BalanceUnbounded}) the exponential factor of $DE_{-}$ with that of $RE$; by solving explicitly with respect to $\mu_{J}$, $h_{J}$ and $N_{J}$ we derive in this case
\[
	\mu_{J} = \frac{1}{t} \bigl(\log \varepsilon - \log \epsilon \bigr) 
	, \quad
	N_{J} = \frac{ \log \varepsilon \sqrt{-\log \epsilon}}{2 \pi \bigl( \sqrt{\bar{\varphi}^{\star}_{J} t} - \sqrt{\log \varepsilon - \log \epsilon} \bigr)}
	, \quad
	h_{J} = \frac{1}{N_{J}} \sqrt{ \frac{\log \epsilon}{\log \epsilon - \log \varepsilon} }
	.
\]

The introduction of the round-off error in the error analysis prevents from placing the contour in a place in which it is not possible to guarantee that round--off errors do not exceed the assigned tolerance $\varepsilon$. Since now an explicit value of $\mu_{J}$ is available, the computation of $\bar{\varphi}^{\star}_{J}$ follows directly from (\ref{eq:Extimation_fbar}) as 
\[
	\bar{\varphi}^{\star}_{J} = \left( \bar{f}^{-\frac{1}{p_{J}}} \sqrt{\mu_{J}} + \sqrt{\varphi(s_{j}^{\star})}  \right)^2 
\]
(obviously, the region $R_{J}$ must be discarded when $\bar{\varphi}^{\star}_{J}$ exceeds the threshold $(\log \varepsilon - \log \epsilon)/t$ since even the accuracy $\bar{f} \varepsilon$ cannot be achieved). 


Since rounding errors depend, in an exponential way, on the value of $t$, it can be useful to scale the ML function in order to force $t$ to assume small values. By simple algebraic manipulations, it is indeed immediate to see that for any $\tau > 0$ 
\[
	e_{\alpha,\beta}(t;\lambda) = \tau^{\beta-1} e_{\alpha,\beta}\left(\frac{t}{\tau};\tau^{\alpha}\lambda\right) 
\]
and it is therefore possible to reduce the propagation of rounding errors by suitably using the above scaling, for instance with $\tau\approx t
$.

\subsection{Extension to three parameters}\label{SS:Prabhakar}

The main information used by the OPC method is the location and the strength of the singularities of the LT; its extension to the 3 parameter ML function (\ref{eq:PrabhakarFunction}) is therefore straightforward. The LT of the corresponding generalization $e_{\alpha,\beta}^{\gamma} (t;\lambda) = t^{\beta-1} E_{\alpha,\beta}^{\gamma} (t^{\alpha} \lambda)$ is indeed
\[
	{\cal E}^{\gamma}_{\alpha,\beta}(s;\lambda) = \frac{s^{\alpha\gamma-\beta}}{(s^{\alpha}-\lambda)^{\gamma}}
	, \quad \Re(s) > 0 \textrm{ and } |\lambda s^{-\alpha}| < 1 ,
\]
which has the same singularities of the 2 parameter counterpart. It is elementary to reformulate Proposition \ref{prop:Mcpuls1} by replacing $\alpha$ with $\alpha \gamma$ and evaluate the new bounds of Propositions \ref{prop:Mcpuls2} and \ref{prop:Mdpuls} respectively as $M_{+}(c_{j}) < \bar{M}_{+} \cdot ( c_{j}^{\star} - c_{j})^{-\gamma}$ and $M_{-}(d_{j}) < \bar{M}_{-} \cdot ( d_{j}^{\star} - d_{j})^{-\gamma}$.

With $\gamma\not=1$ we must restrict the computation to $0<\alpha<1$ and $|\Arg(\lambda)| > \alpha \pi$ since otherwise non trivial difficulties (whose discussion is beyond the scope of the present paper) arise due to more involved branch--cuts; the case $0<\alpha<1$ and $\lambda$ real and negative is however the most interesting for applications \cite{OliveiraMainardiVaz2011}.

\section{Numerical experiments}\label{S:Numerical}

To test the proposed method and verify its computational efficiency we present in this Section some numerical experiments.

All the experiments are performed in Matlab, version 7.9.0.529, on an Intel Dual Core E5400 processor running at 2.70 GHz under the Windows XP operating system; the Matlab code implementing the OPC method and described in the previous sections is made available at \cite{Garrappa2014ml}. As reference we use the values evaluated after summing the series (\ref{eq:ClassicalML}) or (\ref{eq:PrabhakarFunction}) in variable precision arithmetic with 100 digits by means of Maple.

In all the experiments we set the target tolerance $\varepsilon = 10^{-15}$; the goal is to test whether it is possible to provide an approximation $\tilde{E}_{\alpha,\beta}^{\gamma}(z)$ of the ML function $E_{\alpha,\beta}^{\gamma}(z)$ with an accuracy very close to the machine precision. The tolerance $\varepsilon$ represents the absolute error in the computation of the integral in (\ref{eq:GeneralizedMLInverseLaplace2}) and this is the error we expect when the value of the function is not large in modulus (in this case no residue calculation is usually involved); otherwise, the summation of residues can dominate the integral in (\ref{eq:GeneralizedMLInverseLaplace2}) by several orders of magnitude and the leading error is the round--off error in the computation of residues: in the double precision used by Matlab it involves a relative error smaller than  $\varepsilon = 10^{-15}$. The resulting error is therefore a combination of absolute (with small values of $E_{\alpha,\beta}^{\gamma}(z)$) and relative (for large values of $E_{\alpha,\beta}^{\gamma}(z)$) errors and it can be represented as 

\begin{equation}\label{eq:MixedError}
	\frac{\bigl|E_{\alpha,\beta}^{\gamma}(z) - \tilde{E}_{\alpha,\beta}^{\gamma}(z) \bigr|}
			 {1 + \bigl|E_{\alpha,\beta}^{\gamma}(z) \bigr)|} \le 10^{-15}
\end{equation}

In Figure \ref{fig:Fig_Error_OPC_alpha070_beta100_gama100_ang314} we report the error (\ref{eq:MixedError}) for the 2 parameter function $E_{\alpha,\beta}(z)$, for $\alpha=0.7$ and $\beta = 1.0$, evaluated in several points $z$ on the real negative axis. As we can clearly see, the OPC method achieves an accuracy very close to or smaller than the requested tolerance of $10^{-15}$  (the few gaps in the error plot are due to the fact that in some cases the approximated and reference values are exactly the same).

\begin{figure}[ht]
	\centering
	\includegraphics[width=0.65\textwidth]{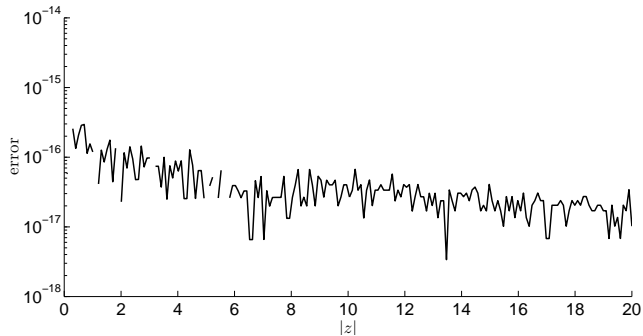} 
	\caption{Error for $E_{\alpha,\beta}(z)$ with $\alpha=0.7$, $\beta=1.0$ and $\arg(z)=\pi$.}
	\label{fig:Fig_Error_OPC_alpha070_beta100_gama100_ang314}
\end{figure}
 
To show the efficiency of the proposed method we present in Figure \ref{fig:Fig_Time_opc_mlf_alpha070_beta100_gama100_ang314.eps} the computational time and we compare it with that of the Matlab {\tt mlf} code \cite{PodlubnyKacenak2012}. This is so far the unique available Matlab code for the ML function and, since it is widely used, it can be considered as a sort of benchmark for testing new methods. 

We observe that whilst the CPU time consumed by OPC remains nearly constant, the {\tt mlf} code demands for a CPU time close or slightly less than OPC for very small and large values of $|z|$ whilst for moderate values of $|z|$ the CPU time of {\tt mlf} is some order of magnitude higher than OPC.

\begin{figure}[ht]
	\centering
	\includegraphics[width=0.65\textwidth]{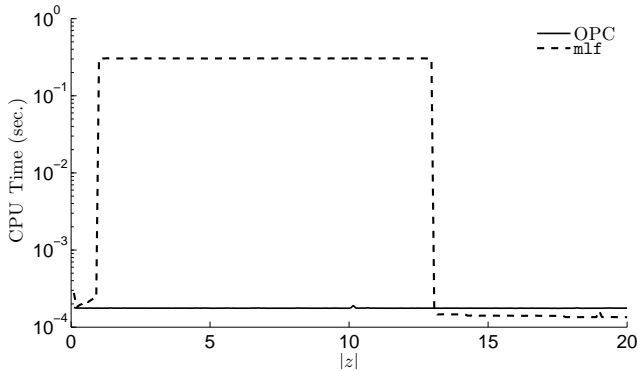} 
	\caption{Computation time for $E_{\alpha,\beta}(z)$ with $\alpha=0.7$, $\beta=1.0$ and $\arg(z)=\pi$.}
	\label{fig:Fig_Time_opc_mlf_alpha070_beta100_gama100_ang314.eps}
\end{figure}

This nonuniform behavior can be explained by observing that {\tt mlf} uses different techniques according to the value of $|z|$: for very small $|z|$ the series (\ref{eq:ClassicalML}) is evaluated until numerical convergence and this computation is quite fast; an asymptotic expansion is instead used when $|z|$ is large and the computation becomes faster and faster as $|z|$ grows; for intermediate values of $|z|$ a Romberg integration is applied to an integral representation of the ML function, with a computational cost proportional to $2^p$ whenever an accuracy $\varepsilon = 10^{-p}$ is requested. On the other hand, most of the computation of OPC is spent by the trapezoidal rule whose cost depends essentially on the number of nodes which is kept at the minimum by the algorithm (and it is roughly proportional to $p$ for any argument $z$); the amount of computation required by the other tasks of OPC, such as location of the singularities, choice of the suitable region and evaluation of the quadrature parameters, is usually negligible.

The plot in Figure \ref{fig:Fig_Error_OPC_alpha050_beta100_gama100_ang157} shows that the OPC algorithm behaves in a robust way and provides results within the requested tolerance also for complex values on the imaginary axis (we used here $\alpha=0.5$, $\beta=1.0$ for which it is known that {\tt mlf} does not provide accurate results).


\begin{figure}[ht]
	\centering
	\includegraphics[width=0.65\textwidth]{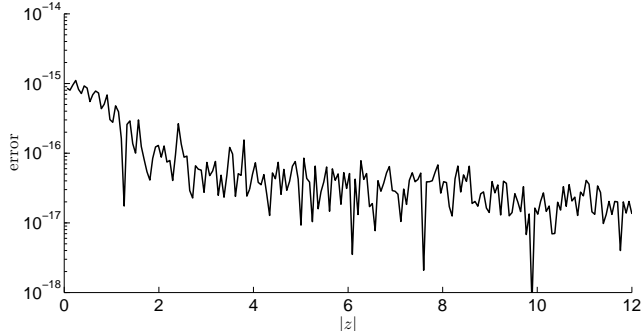} 
	\caption{Error for $E_{\alpha,\beta}(z)$ with $\alpha=0.5$, $\beta=1.0$ and $\arg(z)=\frac{\pi}{2}$.}
	\label{fig:Fig_Error_OPC_alpha050_beta100_gama100_ang157}
\end{figure}

We conclude our experiments by presenting the errors for the three parameter function $E_{\alpha,\beta}^{\gamma}(z)$ for $\alpha=0.6$, $\beta=0.9$, $\gamma=1.2$ and $\arg(z)=\frac{3\pi}{4}$. As we can see from Figure \ref{fig:Fig_Error_OPC_alpha060_beta090_gama120_ang236}, OPC behaves in a satisfactory way and produces errors very close to the target tolerance also for $E_{\alpha,\beta}^{\gamma}(z)$. 

\begin{figure}[ht]
	\centering
	\includegraphics[width=0.65\textwidth]{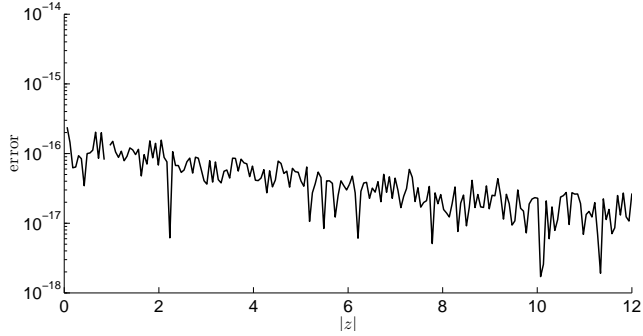} 
	\caption{Error for $E_{\alpha,\beta}^{\gamma}(z)$ with $\alpha=0.6$, $\beta=0.9$, $\gamma=1.2$ and $\arg(z)=\frac{3\pi}{4}$.}
	\label{fig:Fig_Error_OPC_alpha060_beta090_gama120_ang236}
\end{figure}

We do not report the CPU time for $E_{\alpha,\beta}^{\gamma}(z)$ since it would not provide any further information; as discussed in the Subsection \ref{SS:Prabhakar}, the evaluation of the three parameter function just involves different coefficients in the error estimations and most of the computation (and hence the CPU time) is the same as in the two parameter case. 

\section{Concluding remarks}\label{S:Conclusion}

In this work we have presented the OPC method for the evaluation of the two parameter ML function, a function which plays a fundamental role in fractional calculus. The OPC method allows to evaluate the ML function with high accuracy and numerical experiments have shown its computational efficiency. The generalization to the three parameter ML function has been  discussed and tested too. The corresponding Matlab code is made freely available \cite{Garrappa2014ml}.

\section*{Acknowledgments}

The author is extremely grateful to the anonymous referees for their insightful and constructive remarks which allowed to improve the paper in a remarkable way.

\bibliographystyle{siam}
\bibliography{ML3_Biblio}

\begin{thebibliography}{10}

\bibitem{AlBassam1965}
{\sc M.~A. Al-Bassam}, {\em Some existence theorems on differential equations
  of generalized order}, J. Reine Angew. Math., 218 (1965), pp.~70--78.

\bibitem{Butcher1957}
{\sc J.~C. Butcher}, {\em On the numerical inversion of {L}aplace and {M}ellin
  transforms}, in Conference on Data Processing and Automatic Computing
  Machines, Salisbury, Australia, 1957.

\bibitem{OliveiraMainardiVaz2011}
{\sc E.~Capelas~de Oliveira, F.~Mainardi, and J.~Vaz~Jr.}, {\em Models based on
  {M}ittag--{L}effler functions for anomalous relaxation in dielectrics}, Eur.
  Phys. J. Special Topics, 193 (2011), pp.~161--171.

\bibitem{CaputoMainardi1971}
{\sc M.~Caputo and F.~Mainardi}, {\em Linear models of dissipation in anelastic
  solids}, Riv. Nuovo Cimento (Ser. II), 1 (1971), pp.~161--198.

\bibitem{CaputoMainardi2007}
\leavevmode\vrule height 2pt depth -1.6pt width 23pt, {\em A new dissipation
  model based on memory mechanism}, Fract. Calc. Appl. Anal., 10 (2007),
  pp.~309--324.
\newblock Reprinted from Pure Appl. Geophys. {{\bf{9}}1} (1971), no. 1,
  134--147.

\bibitem{DavisRabinowitz1984}
{\sc P.~J. Davis and P.~Rabinowitz}, {\em Methods of numerical integration},
  Computer Science and Applied Mathematics, Academic Press Inc., Orlando, FL,
  second~ed., 1984.

\bibitem{DingfelderWeideman2014}
{\sc B.~Dingfelder and J.A.C. Weideman}, {\em An improved {T}albot method for
  numerical {L}aplace transform inversion}, Numer. Algorithms, 68 (2015),
  pp.~167--183.

\bibitem{DzherbashyanNersesyan1968}
{\sc M.~M. D{\v{z}}rba{\v{s}}jan and A.~B. Nersesjan}, {\em Fractional
  derivatives and the {C}auchy problem for differential equations of fractional
  order}, Izv. Akad. Nauk Armjan. SSR Ser. Mat., 3 (1968), pp.~3--29.

\bibitem{Fox1928}
{\sc C.~Fox}, {\em The asymptotic expansion of integral functions defined by
  generalized hypergeometric functionss}, Proc. London Math. Soc., s2-27
  (1928), pp.~389--400.

\bibitem{Garrappa2014ml}
{\sc R.~Garrappa}, {\em The {M}ittag--{L}effler function}.
\newblock MATLAB Central File Exchange, 2014.
\newblock File ID: 48154.

\bibitem{GarrappaMoretPopolizio2014_JCP}
{\sc R.~Garrappa, I.~Moret, and M.~Popolizio}, {\em Solving the time-fractional
  {S}chr\"{o}dinger equation by {K}rylov projection methods}, J. Comput. Phys.,
  293 (2015), pp.~115--134.

\bibitem{GarrappaPopolizio2011_CAMWA}
{\sc R.~Garrappa and M.~Popolizio}, {\em Generalized exponential time
  differencing methods for fractional order problems}, Comput. Math. Appl., 62
  (2011), pp.~876--890.

\bibitem{GarrappaPopolizio2013}
\leavevmode\vrule height 2pt depth -1.6pt width 23pt, {\em Evaluation of
  generalized {M}ittag--{L}effler functions on the real line}, Adv. Comput.
  Math., 39 (2013), pp.~205--225.

\bibitem{GavrilyukMakarov2005}
{\sc I.~P. Gavrilyuk and V.~L. Makarov}, {\em Exponentially convergent
  algorithms for the operator exponential with applications to inhomogeneous
  problems in {B}anach spaces}, SIAM J. Numer. Anal., 43 (2005),
  pp.~2144--2171.

\bibitem{GorenfloKilbasMainardiRogosin2014}
{\sc R.~Gorenflo, A.~A. Kilbas, F.~Mainardi, and S.~Rogosin}, {\em
  Mittag-Leffler functions. Theory and Applications}, Springer Monographs in
  Mathematics, Springer, Berlin, 2014.

\bibitem{GorenfloLoutchkoLuchko2002}
{\sc R.~Gorenflo, J.~Loutchko, and Y.~Luchko}, {\em Computation of the
  {M}ittag-{L}effler function {$E_{\alpha,\beta}(z)$} and its derivative},
  Fract. Calc. Appl. Anal., 5 (2002), pp.~491--518.

\bibitem{HauboldMathaiSaxena2011}
{\sc H.~J. Haubold, A.~M. Mathai, and R.~K. Saxena}, {\em Mittag-{L}effler
  functions and their applications}, J. Appl. Math.,  (2011), pp.~Art. ID
  298628, 51.

\bibitem{HilferSeybold2006}
{\sc R.~Hilfer and H.~J. Seybold}, {\em Computation of the generalized
  {M}ittag-{L}effler function and its inverse in the complex plane}, Integral
  Transforms Spec. Funct., 17 (2006), pp.~637--652.

\bibitem{HilleTamarkin1930}
{\sc E.~Hille and J.~D. Tamarkin}, {\em On the theory of linear integral
  equations}, Ann. of Math. (2), 31 (1930), pp.~479--528.

\bibitem{KilbasSrivastavaTrujillo2006}
{\sc A.~A. Kilbas, H.~M. Srivastava, and J.~J. Trujillo}, {\em Theory and
  applications of fractional differential equations}, vol.~204 of North-Holland
  Mathematics Studies, Elsevier Science B.V., Amsterdam, 2006.

\bibitem{LopezPalenciaSchadle2006}
{\sc M.~L{\'o}pez-Fern{\'a}ndez, C.~Palencia, and A.~Sch{\"a}dle}, {\em A
  spectral order method for inverting sectorial {L}aplace transforms}, SIAM J.
  Numer. Anal., 44 (2006), pp.~1332--1350.

\bibitem{Mainardi2010}
{\sc F.~Mainardi}, {\em Fractional calculus and waves in linear
  viscoelasticity}, Imperial College Press, London, 2010.

\bibitem{MainardiGorenflo2000}
{\sc F.~Mainardi and R.~Gorenflo}, {\em On {M}ittag-{L}effler-type functions in
  fractional evolution processes}, J. Comput. Appl. Math., 118 (2000),
  pp.~283--299.

\bibitem{Mittag-Leffler1902}
{\sc M.~G. Mittag-Leffler}, {\em Sur l'int\'{e}grale de {L}aplace-{A}bel}, C.
  R. Acad. Sci. Paris (Ser. II), 136 (1902), pp.~937--939.

\bibitem{MittagLeffler1904}
\leavevmode\vrule height 2pt depth -1.6pt width 23pt, {\em Sopra la funzione
  ${E}_{\alpha}(x)$}, Rend. Accad. Lincei, 13 (1904), pp.~3--5.

\bibitem{Moret2013}
{\sc I.~Moret}, {\em A note on {K}rylov methods for fractional evolution
  problems}, Numer. Funct. Anal. Optim., 34 (2013), pp.~539--556.

\bibitem{MoretNovati2011}
{\sc I.~Moret and P.~Novati}, {\em On the convergence of {K}rylov subspace
  methods for matrix {M}ittag--{L}effler functions}, SIAM J. Numer. Anal., 49
  (2011), pp.~2144--2164.

\bibitem{MurliRizzardi1990}
{\sc A.~Murli and M.~Rizzardi}, {\em Algorithm 682: {T}albot's method of the
  {L}aplace inversion problems}, ACM Trans. Math. Softw., 16 (1990),
  pp.~158--168.

\bibitem{Podlubny1999}
{\sc I.~Podlubny}, {\em Fractional differential equations}, vol.~198 of
  Mathematics in Science and Engineering, Academic Press Inc., San Diego, CA,
  1999.

\bibitem{PodlubnyKacenak2012}
{\sc I.~Podlubny and M.~Kacenak}, {\em The {M}atlab mlf code}.
\newblock MATLAB Central File Exchange, 2001--2012.
\newblock File ID: 8738.

\bibitem{Prabhakar1971}
{\sc T.~R. Prabhakar}, {\em A singular integral equation with a generalized
  {M}ittag--{L}effler function in the kernel}, Yokohama Math. J., 19 (1971),
  pp.~7--15.

\bibitem{HilferSeybold2008}
{\sc H.~Seybold and R.~Hilfer}, {\em Numerical algorithm for calculating the
  generalized {M}ittag-{L}effler function}, SIAM J. Numer. Anal., 47 (2008/09),
  pp.~69--88.

\bibitem{SheenSloanThomee2003}
{\sc D.~Sheen, I.~H. Sloan, and V.~Thom{\'e}e}, {\em A parallel method for time
  discretization of parabolic equations based on {L}aplace transformation and
  quadrature}, IMA J. Numer. Anal., 23 (2003), pp.~269--299.

\bibitem{Talbot1979}
{\sc A.~Talbot}, {\em The accurate numerical inversion of {L}aplace
  transforms}, J. Inst. Math. Appl., 23 (1979), pp.~97--120.

\bibitem{TrefethenWeideman2014}
{\sc L.~N. Trefethen and J.~A.~C. Weideman}, {\em The exponentially convergent
  trapezoidal rule}, SIAM Rev., 56 (2014), pp.~385--458.

\bibitem{TrefethenWeidemanSchmelzer2006}
{\sc L.~N. Trefethen, J.~A.~C. Weideman, and T.~Schmelzer}, {\em Talbot
  quadratures and rational approximations}, BIT, 46 (2006), pp.~653--670.

\bibitem{Tricomi1954}
{\sc F.~G. Tricomi}, {\em Funzioni ipergeometriche confluenti}, Edizione
  Cremonese, Roma, 1954.

\bibitem{Weideman2006}
{\sc J.~A.~C. Weideman}, {\em Optimizing {T}albot's contours for the inversion
  of the {L}aplace transform}, SIAM J. Numer. Anal., 44 (2006), pp.~2342--2362.

\bibitem{Weideman2010}
\leavevmode\vrule height 2pt depth -1.6pt width 23pt, {\em Improved contour
  integral methods for parabolic {PDE}s}, IMA J. Numer. Anal., 30 (2010),
  pp.~334--350.

\bibitem{WeidemanTrefethen2007}
{\sc J.~A.~C. Weideman and L.~N. Trefethen}, {\em Parabolic and hyperbolic
  contours for computing the {B}romwich integral}, Math. Comp., 76 (2007),
  pp.~1341--1356.

\bibitem{Wiman1905}
{\sc A.~Wiman}, {\em \"{U}ber den fundamental satz in der teorie der funktionen
  {${E}_{\alpha}(x)$}}, Acta Math., 29 (1905), pp.~191--201.

\bibitem{WolframHypergeometricU}
{\sc {Wolfram Research Inc.}}, {\em Tricomi confluent hypergeometric function},
  1998--2010.

\bibitem{Wright1935}
{\sc E.~M. Wright}, {\em The asymptotic expansion of the generalised
  hypergeometric function}, J. London Math. Soc., s1-10 (1935), pp.~286--293.

\end{thebibliography}

\end{document}